\documentclass[11pt]{amsart}
\usepackage{geometry}                % See geometry.pdf to learn the layout options. There are lots.
\usepackage{graphicx}
\usepackage{amsmath}%
\usepackage{amsfonts}%
\usepackage{amssymb}
\usepackage{epstopdf}
\usepackage{color}
\usepackage{esint}
\usepackage{hyperref}
\usepackage{cite}
\input xy
\xyoption{all}

\newtheorem{theorem}{Theorem}
\theoremstyle{plain}

\newtheorem{lemma}{Lemma}
\newtheorem{proposition}{Proposition}

\theoremstyle{definition}
\newtheorem{definition}{Definition}
\theoremstyle{remark}
\newtheorem{example}{Example}

\newtheorem{remark}{Remark}

\numberwithin{equation}{section}

\newcommand{\mbb}{\mathbb}
\newcommand{\mbf}{\mathbf}
\newcommand{\mc}{\mathcal}
\newcommand{\on}{\operatorname}
\newcommand{\Lied}{\mathcal{L}}

\title{Integration of Exact Courant Algebroids}

\author{David Li-Bland}
\author{Pavol \v{S}evera}

\address{Department of Mathematics, University of Toronto, 40 St George Street, Toronto, Ontario
M4S2E4, Canada}
\email{dbland@math.toronto.edu}
\address{Department of Mathematics, Universit\'{e} de Gen\`{e}ve, Geneva, Switzerland, on leave
from Dept. of Theoretical Physics, FMFI UK, Bratislava, Slovakia}
\email{pavol.severa@gmail.comphs}
\begin{document}
\begin{abstract}
In this paper, we describe an integration of exact Courant algebroids to symplectic 2-groupoids, %We show how to lift gauge transformations by two forms ($B$-field transforms) to the symplectic 2-groupoids. 
and we show that the differentiation procedure from \cite{Severa:2007vu} inverts our integration.
\end{abstract}
\maketitle
\tableofcontents
\section{Introduction}
A Courant algebroid is a Lie 2-algebroid paired with a compatible symplectic structure \cite{Roytenberg:1998ku,Roytenberg99,Roytenberg02,Roytenberg:2002,LetToWein,Severa:2005vla}. Therefore, in the study of Courant algebroids %an immediate question is, what is the global object ``integrating'' it? Secondly, how do you ``differentiate'' the global object to recover the Courant algebroid?
two immediate questions are
\begin{description}
\item[Question 1] What is the global object ``integrating'' a Courant algebroid?
\item[Question 2] How do you ``differentiate'' the global object to recover the Courant algebroid?
\end{description}

Question~1 was first addressed in \cite{Severa:2005vla}, where, given a Courant algebroid, a construction of a 2-groupoid carrying a symplectic form on the 2-simplices is sketched (see also \cite{Sheng:2009un}). Unfortunately, this construction was infinite dimensional.  A procedure to differentiate Lie $n$-groupoids to Lie $n$-algebroids is described in \cite{Severa:2007vu}, giving an  answer to Question~2. However, this differentiation procedure does not invert the integration procedure described in \cite{Severa:2005vla} ``on the nose''.

Although the answer to the Question~1 might be somewhat complicated, it has a simple solution for some of the most popular Courant algebroids, namely exact Courant algebroids \cite{Severa:2001}. For instance, an elegant solution for exact Courant algebroids with trivial characteristic class is presented in \cite{Mehta:2010ue} in terms of the ``bar'' construction. In this paper, we present a solution for arbitrary exact Courant algebroids. Our construction results in a (local) Lie 2-groupoid carrying a symplectic form on the space of 2-simplices. Furthermore, we show that if one differentiates our construction (as in \cite{Severa:2007vu}), one recovers the original Courant algebroid.

The idea of our construction is as follows. As a differential graded manifold, the standard Courant algebroid over a manifold $M$ is $T[1]T^*[1]M$. Since it is of the form $T[1]X$ for some graded manifold $X$ ($X=T^*[1]M$), the problem of its integration to a local  2-groupoid $K_\bullet$ has a simple solution (see Remark~\ref{rem:simGlbInt}). Next we construct a symplectic form on $K_2$ making $K_\bullet$ into a (local) 2-symplectic 2-groupoid which differentiates to the standard Courant algebroid.

We then prove that an arbitrary exact Courant algebroid over $M$ is isomorphic, as a differential graded manifold, to the standard one. We can therefore use the same 2-groupoid $K_\bullet$ as its integration. The isomorphism is not, however,  a symplectomorphism. We compute the modification to the symplectic form on $T[1]T^*[1]M$ and modify the symplectic form on $K_2$ correspondingly.

In \cite{Sheng:2011vz}, a more conceptual explanation for our construction is given, related to the work of Mehta, Gracia-Saz, Arias Abad, Crainic and Schaetz on actions up to homotopy \cite{gracia2010lie,GraciaSaz:2010vm,abad2009representations,arias2010a_infty} (see also \cite{Sheng:2009un}).

%As presented in this paper, our construction seems unmotivated I'M NOT SURE ABOUT THIS. This gap will be filled in \cite{Sheng:2011vz}, where a more conceptual construction is given, and it is related to the work of Mehta, Gracia-Saz, Camillo, Crainic and Schaetz on actions up to homotopy \cite{gracia2010lie,gracia2010vb,abad2009representations,arias2010a_infty} (see also \cite{Sheng:2009un}).
\subsubsection{Acknowledgements}
D.L.-B was supported by an NSERC CGS-D grant, and would like to thank Eckhard Meinrenken for his advice and suggestions, and Chenchang Zhu, and Rajan Mehta for helpful discussions. P. \v{S}. was supported by the Swiss National Science Foundation (grant 200020-120042 and 200020-126817).

We also thank the referee for their valuable recommendations which have made this paper considerably easier to read.

\section{Background}
\begin{remark}[A note on Lie groupoids and Lie algebroids]
%In this paper, w
We use the definitions of (local) Lie $n$-groupoids given by Andre Henriques and Chenchang Zhu \cite{Henriques:2008cda,zhu2009kan}, in terms of simplicial manifolds.
We take the definition of Lie $n$-algebroids given in \cite{Severa:2007vu} in terms of $NQ$-manifolds.
\end{remark}

In this section, we recall the differentiation procedure described in \cite{Severa:2007vu} which takes Lie $n$-groupoids to Lie $n$-algebroids. 
%The procedure is given by a functor 
%from simplicial manifolds to pre-sheaves of super-manifolds, \begin{equation}\label{eq:1-Jet}\on{1-Jet}:\on{Man}^{{\mbf{\Delta}}^{\on{op}}}\to\widehat{\on{SMan}}.\end{equation}

%It was shown in \cite{Severa:2007vu} that when restricted to Lie $n$-groupoids, $\on{1-Jet}$ takes values in the category of $NQ$-manifolds, $\on{SMan}_{NQ}$. Consequently, \eqref{eq:1-Jet} can be thought of as differentiation from Lie $n$-groupoids to Lie $n$-algebroids.

\subsection{Simplicial Manifolds}
For $n\in\mbb{N}$, let $[n]$ be the category generated by the directed graph $$0\to1\to\cdots\to n,$$ and ${\mbf{\Delta}}$ the full subcategory of $\on{Cat}$ (the category of small categories) generated by the objects $[0],[1],[2],\dots$. We have the distinguished functors \cite{goerss2009simplicial} $$s^j(0\to1\to\cdots\to n+1)=(0\to 1\to \cdots\to j\xrightarrow{\on{id}}j\to\cdots\to n),$$ (that is, we insert the identity in the $j^{th}$ place), and $$d^i(0\to1\to\cdots\to n-1)=(0\to1\to\cdots\to i-1\to i+1\to\cdots\to n),$$ (that is, we compose $i-1\to i\to i+1$). We denote the corresponding maps in ${\mbf{\Delta}}^{op}$ by $s_j$ and $d_i$.

Let $\on{Man}$ be the category of smooth manifolds. Then $\on{Man}^{{\mbf{\Delta}}^{\on{op}}}$ is the category of simplicial manifolds.  More generally, if $C$ is any category, then $C^{{\mbf{\Delta}}^{\on{op}}}$ is the category of simplicial objects (in $C$). 

	As a quick word on notation, if $X\in C^{{\mbf{\Delta}}^{\on{op}}}$, then it is conventional to write $X_n:=X([n])$.

\begin{example}\label{ex:StdSimplices}
The {\em standard $n$-simplex}, $\Delta^n$, is the contravariant functor $$\Delta^n:=\on{Cat}(\cdot,[n])\in\on{Set}^{{\mbf{\Delta}}^{\on{op}}}\subset \on{Man}^{{\mbf{\Delta}}^{\on{op}}}$$ (where sets are viewed as discrete manifolds).
\end{example}

\begin{example}\label{ex:Horn}
The $k^{th}$ horn $\Lambda_k^n\subset \Delta^n$ is defined as
$$(\Lambda_k^n)([j])=\{f\in \Delta^n([j])\mid d^k([n])\nsubseteq f([j])\}.$$ We may think of $\Lambda^n_k$ as the boundary of the standard $n$-simplex $\Delta^n$ with the $k^{th}$ face removed. 
\end{example}

Let $M_\bullet$ be a simplicial manifold. The natural inclusion $\Lambda_k^n\subset \Delta^n$ induces a map \begin{equation}\label{eq:hornFilling}\on{Man}^{{\mbf{\Delta}}^{\on{op}}}(\Delta^n,M_\bullet)\to\on{Man}^{{\mbf{\Delta}}^{\on{op}}}(\Lambda^n_k,M_\bullet).\end{equation}
The {\em Kan conditions} \cite{Henriques:2008cda} are 
$$Kan(n,k):\text{ \eqref{eq:hornFilling} is a surjective submersion},\quad Kan!(n,k):\text{ \eqref{eq:hornFilling} is a diffeomorphism}.$$

%and the {\em local Kan conditions} \cite{zhu2009kan} are
%$$Kan^l(n,k):\text{ \eqref{eq:hornFilling} is a submersion},\quad Kan^l!(n,k):\text{ \eqref{eq:hornFilling} is injective \'{e}tale}.$$
The following definition is due to Henriques \cite{Henriques:2008cda}.
\begin{definition}[\cite{Henriques:2008cda}]\label{def:liengrpoid} A {\em Lie $n$-groupoid} is a simplicial manifold $M_\bullet$ satisfying 
\begin{itemize}
\item $Kan!(m,k)$ for $m>n$ and $0\leq k\leq m$, and
\item $Kan(m,k)$ for $m> 0$ and $0\leq k\leq m$.
\end{itemize}
\end{definition}

In particular, a Lie $1$-groupoid is the nerve of a Lie groupoid. %The simplicial manifold $E_\bullet M$ (for any $M\in\on{Man}$) defined in Example~\ref{ex:EM} is a Lie $1$-groupoid (it is the nerve of the pair groupoid).
For instance, the nerve of the pair groupoid (defined explicitly in the following example) is a Lie $1$-groupoid.

\begin{example}\label{ex:EM}
There is a full and faithful functor $E:\on{Man}\to\on{Man}^{{\mbf{\Delta}}^{\on{op}}}$ given by $E_nM=M^{n+1}$ and
\begin{equation*}\label{eq:EM}EM(f):(z_0,\dots,z_m)\to(z_{f(0)},\dots,z_{f(n)}),\end{equation*}
 for any monotone map $f:[n]\to[m]$.
 \end{example}

We will be interested in the local version of $n$-groupoids, as introduced by Zhu \cite{zhu2009kan}. However, as suggested by a referee, we will reformulate it in terms of the microfolds introduced by Cattaneo,  Dherin, and  Weinstein \cite{Cattaneo:2010uq} and  Blohmann, Fernandes, and  Weinstein \cite{Blohmann:2010ub} following Milnor \cite{Milnor:1964tm}.

\begin{definition}\label{def:microfold}
A \emph{microfold} is an equivalence class of pairs $(M,S)$ of manifolds such that $S\subseteq M$ is a closed submanifold. Two such pairs $(M_1,S_1)$ and $(M_2,S_2)$ are said to be equivalent if $S_1=S_2=S$ and there exists a third pair $(U,S)$ such that $U$ is simultaneously an open subset of both $M_1$ and $M_2$. We denote the equivalence class by $[M,S]$, and refer to $S$ as the \emph{microfold core} of $[M,S]$.

A morphism  between microfolds is a germ of maps between representatives which takes the source microfold core to the target microfold core. Such a morphism is said to be a \emph{surjective submersion} (resp. a \emph{diffeomorphism}) if it is a surjective submersion (resp. a diffeomorphism) for a suitable choice of representatives.
\end{definition}

We denote the category of microfolds by $\on{Mfold}$. There is a forgetful functor $F_{core}:\on{Mfold}\to\on{Man}$ which takes a microfold $[M,S]$ to its microfold core $S$. %It is the right adjoint to the inclusion functor $I:\on

\begin{definition}\label{def:locSimpMan}The category $\on{Man}_{loc}^{{\mbf{\Delta}}^{\on{op}}}$ of \emph{local simplicial manifolds} (or \emph{simplicial manifold germs}) is the subcategory of functors $[M_\bullet,S_\bullet]:\mbf{\Delta}^{\on{op}}\to\on{Mfold}$ such that 
\begin{itemize}
\item $M_0=S_0$, and
\item the composition $F_{core}\circ [M_\bullet,S_\bullet]:=S_\bullet$ with the forgetful functor is a constant functor.
\end{itemize}

We will denote a local simplicial manifold $[M_\bullet,S_\bullet]:=[M_\bullet,M_0]$ simply by $[M_\bullet]$.
\end{definition}

Working with microfolds in place of manifolds, one obtains direct analogues of Henriques'  Kan conditions \cite{Henriques:2008cda} for local simplicial manifolds. Thus one obtains the following definition of a local Lie $n$-groupoid, essentially a reformulation of the one by Zhu \cite{zhu2009kan}.

%The {\em Kan conditions} \cite{Henriques:2008cda,zhu2009kan} are 
%$$Kan(n,k):\text{ \eqref{eq:hornFilling} is a surjective submersion},\quad Kan!(n,k):\text{ \eqref{eq:hornFilling} is a diffeomorphism}.$$

\begin{definition}\label{def:locliengrpoid} A {\em local Lie $n$-groupoid} is a local simplicial manifold $[M_\bullet]$ satisfying 
\begin{itemize}
\item $Kan^l!(m,k)$ for $m>n$ and $0\leq k\leq m$, and
\item $Kan^l(m,k)$ for $m> 0$ and $0\leq k\leq m$.
\end{itemize}
\end{definition}

There is a forgetful functor $[\cdot]:\on{Man}^{{\mbf{\Delta}}^{\on{op}}}\to \on{Man}_{loc}^{{\mbf{\Delta}}^{\on{op}}}$ which takes a simplicial manifold $M_\bullet$ to its germ $[M_\bullet]:=[M_\bullet,M_0]$. It restricts to a functor from Lie $n$-groupoids to local Lie $n$-groupoids.

\subsection{$NQ$-manifolds}
%A good way of defining Lie $n$-algebroids is in terms of $NQ$-manifolds \cite{Severa:2007vu}.
An $NQ$-manifold is a differential non-negatively graded manifold. We recall a reformulation of this definition from \cite{Severa:2007vu}.

We let $\on{SMan}$ denote the category of super-manifolds.
Let $\theta$ be the standard coordinate on the odd line $\mbb{R}^{0\mid 1}$. A general map $$\mbb{R}^{0\mid 1}\to\mbb{R}^{0\mid 1},\quad \theta\to \xi+s\theta$$  can be identified with an element $(s,\xi)\in\mbb{R}^{1\mid 1}$.
The corresponding super-semi-group $\on{\underline{End}}(\mbb{R}^{0\mid 1}):=\on{\underline{SMan}}(\mbb{R}^{0\mid 1},\mbb{R}^{0\mid 1})\cong \mbb{R}^{1\mid 1}$ carries the multiplication $$(s,\xi)\cdot (t,\eta)=(st,\xi+s\eta),\quad s,t\in\mbb{R},\;\xi,\eta\in\mbb{R}^{0\mid 1}.$$ 

\begin{definition}[\cite{Severa:2005vla,kontsevich03,Kochan03}]\label{def:NQmfld}
An $NQ$-manifold \cite{Severa:2005vla,kontsevich03,Kochan03} is a super-manifold $X\in\on{SMan}$ with an action of $\on{\underline{End}}(\mbb{R}^{0\mid 1})$ such that $(-1,0)$ acts as the parity operator (it just changes the sign of the odd coordinates). In this context, a function $f\in C^\infty(X)$ has degree $\lvert f\rvert:=k$ if $$(s,0)\cdot f=s^k f.$$ Note that the degree of a function is always a non-negative integer. Furthermore, the homological vector field $Q$ on $X$ is defined by $$Q:f\to \frac{\partial}{\partial \xi}\big((0,\xi)\cdot f\big).$$

The manifold $(0,0)\cdot X\subset X$ is called the base of  $X$.
% The map $B:X\to (0,0)\cdot X$ is a functor from $\on{SMan}_{NQ}$ to $\on{Man}$. 

An $N$-manifold is a super-manifold $X\in\on{SMan}$ with an action of the multiplicative semi-group $\mbb{R}$ such that $(-1,0)$ acts as the parity operator.

Finally, a bi-$NQ$-manifold is a super-manifold $X\in\on{SMan}$ with an action of $\big(\on{\underline{End}}(\mbb{R}^{0\mid 1})\big)^2$ such that $\big((-1,0),(-1,0)\big)$ acts as the parity operator. A function $f\in C^\infty(X)$ has degree $\lvert f\rvert:=(k,l)$ if $\big((s,0),(t,0)\big)\cdot f=s^kt^l f$. Additionally, $X$ is said to be {\em concentrated in the  second grading} if $\big((0,0),(1,0)\big)$ acts by the identity on $X$.
\end{definition}

We let $\on{SMan}_{NQ}$ and $\on{SMan}_N$ denote the categories of $NQ$-manifolds and $N$-manifolds, respectively. Definitions~\ref{def:microfold} and \ref{def:locSimpMan} also extend in the obvious way to define the categories $\on{SMfold}_{NQ}$, $\on{SMfold}_N$, $\big(\on{SMan}_{NQ}\big)_{loc}^{\mbf{\Delta}^{\on{op}}}$ and $\big(\on{SMan}_N\big)_{loc}^{\mbf{\Delta}^{\on{op}}}$ of $NQ$-microfolds, $N$-microfolds, and local simplicial $NQ$ and $N$-manifolds, respectively.

\begin{example}\label{ex:T[1]M}
If $M$ is any manifold, then the pre-sheaf $\on{SMan}(\cdot\times\mbb{R}^{0\mid1},M)$ is represented by the super-manifold $T[1]M$. The natural action of $\on{\underline{End}}(\mbb{R}^{0\mid 1})$ on $\on{SMan}(\cdot\times\mbb{R}^{0\mid1},M)$ is such that $(-1,0)$ acts as the parity operator. Therefore $T[1]M$ is an $NQ$-manifold.

In fact $C^\infty(T[1]M)=\Omega^\bullet(M)$ and the homological vector field on $T[1]M$ is just the de Rham differential $Q_{deRham}=d$.
\end{example}

\subsection{The Functor $\on{1-Jet}:\on{Man}_{loc}^{{\mbf{\Delta}}^{\on{op}}}\to\widehat{\on{SMan}}$} 
In this section, we recall the functor $\on{1-Jet}$, which maps the category of local simplicial manifolds, $\on{Man}_{loc}^{{\mbf{\Delta}}^{\on{op}}}$, to the category of presheaves of supermanifolds, $\widehat{\on{SMan}}:=\on{Sets}^{({\on{SMan}}^{\on{op}})}$ (where $\on{Sets}$ is the category of sets).

Suppose that $[M_\bullet]\in \on{Man}_{loc}^{{\mbf{\Delta}}^{\on{op}}}$, then (following \cite{Severa:2007vu}) we define 
\begin{equation}\label{eq:1jetFrm}\on{1-Jet}([M_\bullet])(X)=\on{SMfold}^{{\mbf{\Delta}}^{\on{op}}}([X\times E_\bullet\mbb{R}^{0\mid 1}],[M_\bullet]), \text{ for any } X\in\on{SMan}.\end{equation}
 Since  $E_\bullet\mbb{R}^{0\mid 1}$ carries a natural action of $\on{\underline{End}}(\mbb{R}^{0\mid 1})$, so does $\on{1-Jet}([M_\bullet])\in\widehat{\on{SMan}}$. %In fact \cite{Severa:2007vu}, $\on{1-Jet}(M_\bullet)\in\widehat{\on{SMan}_{NQ}}$. 
Furthermore, we have
\begin{proposition}[\cite{Severa:2007vu}] If $[M_\bullet]$ is a local Lie $n$-groupoid, then $\on{1-Jet}([M_\bullet])$ is representable, therefore $\on{1-Jet}([M_\bullet])\in\on{SMan}_{NQ}$.\end{proposition}
%(Where we refer the reader to \cite{Henriques:2008cda} for a definition of Lie $n$-groupoids.)
\begin{remark}
Note that $\on{1-Jet}$ is not a faithful functor. In general, for $n>1$, there are \emph{more} morphisms between local Lie $n$-groupoids than between their corresponding $NQ$-manifolds (see Remark~\ref{rem:nonUniqueSymplForm} for a relevant example). Furthermore, the composition $$M_\bullet\to[M_\bullet]\to\on{1-Jet}([M_\bullet])$$ from Lie $n$-groupoids to $NQ$-manifolds already fails to be full for $n=1$ (see the work Crainic and Fernandes \cite{Lie-Algebroids}).
\end{remark}

\begin{definition}\label{def:integrates}
We say that a local Lie $n$-groupoid $[M_\bullet]$ {\em integrates} an $NQ$-manifold $X$ if $\on{1-Jet}([M_\bullet])\cong X$. Similarly, we say that a Lie $n$-groupoid $M_\bullet$ {\em integrates} an $NQ$-manifold $X$ if $\on{1-Jet}([M_\bullet])\cong X$.
\end{definition}

\begin{example}\label{ex:EMintTM}
Since the functor $E:\on{Man}\to\on{Man}^{{\mbf{\Delta}}^{\on{op}}}$ from Example~\ref{ex:EM} is full and faithful, $E_\bullet M$ integrates $T[1]M$.
\end{example}

\begin{example}
If $G$ is a Lie group and $\mathfrak{g}$ is its Lie algebra, then the nerve of $G$ integrates the $NQ$-manifold $\mathfrak{g}[1]$. More generally, if $\Gamma$ is a Lie groupoid and $A$ the corresponding Lie algebroid, then (the nerve of) $\Gamma$ integrates $A[1]$.
\end{example}

\begin{example}\label{ex:DoldKan} Let $\on{Vect}_f(\mbb{R})$ denote the category of finite dimensional vector spaces over $\mbb{R}$, $\on{Ch}^+(\on{Vect}_f(\mbb{R}))$ the positively graded chain complexes and $V\in\on{Vect}_f(\mbb{R})$.
%Let $V$ be a vector space over $\mbb{R}$, viewed as a chain complex concentrated in degree 0. 
The Dold-Kan correspondence
$$N:\big(\on{Vect}_f(\mbb{R})\big)^{{\mbf{\Delta}}^{\on{op}}}\rightleftarrows \on{Ch}^+(\on{Vect}_f(\mbb{R})):\Gamma$$
 defines the Eilenberg-Mac Lane object $K(V,n):=\Gamma(V[-n])$, which we view as a simplicial manifold. 
 
 Concretely, \begin{equation}\label{eq:EMLOb}\begin{array}{rl}
 K(V,n)_k:=\big\{\{v_f\}_{f:[n]\to[k]}\mid&v_f\in V,\; v_f=0 \text{ if } f \text{ is not injective,}\\ 
 & \text{and for every } g:[n+1]\to [k], \;\sum_{i=0}^{n+1}(-1)^i v_{g\circ d^i}=0\big\}
 \end{array}\end{equation}
 and for any monotone map $h:[l]\to[k]$, 
 $$K(V,n)(h): K(V,n)_k\to K(V,n)_l\quad \bigg(\{v_f\}_{f:[n]\to[k]}\to \{v'_{f'}:=v_{h\circ f'}\}_{f':[n]\to[l]}\bigg).$$
 
 The following picture can be useful. An element of $K(V,n)_k$ is a labelling of the $n$ dimensional faces of the standard $k$ simplex by elements of $V$, so that the alternating sum around any $n+1$ dimensional face is zero. Note that we have a diffeomorphism $V\cong K(V,n)_n$ given by \begin{equation}\label{eq:KVnn}v\to \left\{ v_f=\left\{\begin{array}{rl} v&\text{ if }f=\on{id}:[n]\to[n]\\0&\text{ otherwise.}\end{array}\right\}\right\}_{f:[n]\to[n]}\end{equation}
%\begin{figure}
%\def\svgwidth{\columnwidth}
%\input{KV2.pdf_tex}
%\caption{\label{fig:KV2} The 2-groupoid $\mc{T}^*_\bullet M$. 
%(Here $z_i\in M$ and $w_{i,j}\in W_{(z_i,z_j)}$.)}
%\end{figure}

$K(V,n)$ integrates the $NQ$-manifold $V[n]$ (with trivial $Q$-structure). For any $v\in V[n]$, the corresponding map $E_k\mbb{R}^{0\mid1}\to K(V,n)_k$ is given by \begin{equation}\label{eq:VEKVmap}(\theta_0,\dots,\theta_k)\to \big\{v_f:=v\big({\textstyle \sum_{i=0}^n(-1)^i\theta_{f\circ d^i(0)}\cdots\theta_{f\circ d^i(n-1)}}\big)\big\}_{f:[n]\to[k]}.\end{equation}
\end{example}

\subsubsection{Multiplicative Forms}

Following \cite{Crainic:2003bt,Crainic02,Bursztyn03-1,SCattaneo:2004fe, Ponte:2005txa,Bursztyn09,Bursztyn:2010wb}
we make the following definition of multiplicative forms on a simplicial manifold.

\begin{definition}\label{def:MultForm1}
 Let $M_\bullet$ be a (local) simplicial manifold. We say that a $k$-form $\alpha\in\Omega^k(M_n)$ is {\em multiplicative} if %it defines a map \begin{equation}\label{eq:multkform}\alpha:T[1]M_\bullet\to K(\mbb{R}[k],n).\end{equation}
%TODO: EXPLAIN BETTER? NOT THE SAME $\alpha$
% Equivalently,
for any $0\leq i<n$, $s_i^*\alpha=0$, and 
  $D\alpha=0$, where \begin{equation}\label{eq:simpDiff}D:=\sum_{i=0}^n (-1)^i d_i^*:\Omega^k(M_n)\to \Omega^k(M_{n+1})\end{equation} is the simplicial differential.
\end{definition}

In the spirit of \cite{Mackenzie-Xu94,Xu95,Mackenzie97,Bursztyn09,Bursztyn:2010wb} we would like to interpret Definition~\ref{def:MultForm1} in terms of morphisms of simplicial manifolds. But first we need to point out that one can extend the functor $\on{1-Jet}$ to (local) simplicial $NQ$-manifolds,
$$\on{1-Jet}:\big(\on{SMan}_{NQ}\big)_{loc}^{{\mbf{\Delta}}^{\on{op}}}\to\widehat{\on{SMan}},$$ by the same formula \eqref{eq:1jetFrm}. 

If $X_\bullet$ is a (local) simplicial $NQ$-manifold, then $\on{1-Jet}([X_\bullet])$ carries an action of $\big(\on{\underline{End}}(\mbb{R}^{0\mid 1})\big)^2$. Specifically, in the formula
$$\on{1-Jet}([X_\bullet])(Z)=\on{SMan}^{{\mbf{\Delta}}^{\on{op}}}([Z\times E_\bullet\mbb{R}^{0\mid 1}],[X_\bullet]), \text{ (for any } Z\in\on{SMan}),$$
 $\on{\underline{End}}(\mbb{R}^{0\mid 1})\times (1,0)$ acts directly on the factor $X_\bullet$ and $(1,0)\times \on{\underline{End}}(\mbb{R}^{0\mid 1})$ acts directly on the factor $\mbb{R}^{0\mid 1}$.% in the formula 
%$$\on{1-Jet}(X_\bullet)(Z)=\on{SMan}^{{\mbf{\Delta}}^{\on{op}}}(Z\times E_\bullet\mbb{R}^{0\mid 1},X_\bullet), \text{ (for any } Z\in\on{SMan}).$$

Just as certain (local) simplicial manifolds integrate $NQ$-manifolds, certain (local) simplicial $NQ$-manifolds integrate bi-$NQ$-manifolds. 

%(it then takes values in pre-sheaves of bigraded simplicial manifolds).
 We will be interested in the following example.

\begin{example}
Suppose that $M_\bullet$ is a (local) simplicial manifold integrating the $NQ$-manifold $X$. View $X$ as a bi-$NQ$-manifold concentrated in the second grading. Then $T[1]M_\bullet$ integrates $T[1,0]X$. (This follows immediately from the definition of $\on{1-Jet}$ and the fact that $T[1,0]N$ represents the pre-sheaf $\on{SMan}(\cdot\times \mbb{R}[-1,0],N)$ for any $N\in\on{SMan}$).
\end{example}

Recall that a $k$-form $\alpha\in\Omega^k(M)$ can be thought of as a (grading-preserving) function $\alpha:T[1]M\to \mbb{R}[k]$. 
 %Furthermore, since the identity map $\on{id}:[n]\to[n]$ is the only injective map $[n]\to[n]$, there is a canonical identification $K(\mbb{R}[k],n)_n\cong\mbb{R}[k]$.

\begin{lemma}
Let $M_\bullet$ be a (local) simplicial manifold, and $\alpha$ a $k$-form $\alpha\in\Omega^k(M_n)$.
The corresponding map \begin{equation}\label{eq:prealph}\alpha:T[1]M_n\to\mbb{R}[k]\cong K(\mbb{R}[k],n)_n\end{equation}
extends to a simplicial map to the Eilenberg-Mac Lane object \eqref{eq:EMLOb},
\begin{equation}\label{eq:hatalpha1}\hat\alpha:T[1]M_\bullet\to K(\mbb{R}[k],n)_\bullet,\end{equation}
if and only if $\alpha$ is multiplicative. In this case, the extension \eqref{eq:prealph} is unique.
\end{lemma}
\begin{proof}
First, we will compute the unique extension $\hat\alpha$ (assuming it exists).

 Using \eqref{eq:KVnn}, we see that for $x\in T[1]M_n$, \eqref{eq:prealph} is given by $$\hat\alpha(x)=\left\{v'_{f'}=\left\{\begin{array}{rl} \alpha(x) & \text{ if }f':=\on{id}\\ 0&\text{ otherwise.}\end{array}\right\}\right\}_{f':[n]\to[n]}$$ 
 
Let $y\in T[1]M_l$, and suppose $\hat\alpha(y)=\{v_f\}_{f:[n]\to [l]}$. Then for an arbitrary map $g:[n]\to [l]$, 
\begin{multline}
\{v'_{f'}:=v_{g\circ f'}\}_{f':[n]\to[n]}=K(\mbb{R}[k],n)(g)\circ\hat\alpha (y)
\\=\hat\alpha\circ T[1]M_\bullet(g)(y)=\left\{v'_{f'}:=\left\{\begin{array}{rl} \alpha\circ T[1]M_\bullet(g)(y) & \text{ if }f'=\on{id}\\ 0&\text{ otherwise.}\end{array}\right\}\right\}_{f':[n]\to[n]}
\end{multline}

Therefore $v_g=\alpha\circ T[1]M_\bullet(g)(y)$. It follows that $\hat\alpha$ must be defined by
\begin{equation}\label{eq:hatalpha2}\hat\alpha(y)=\{v_f:=\alpha\circ T[1]M_\bullet(f)(y)\}_{f:[n]\to[l]},\end{equation} for any $y\in T[1]M_l$. We let $v_f(y):=\alpha\circ T[1]M_\bullet(f)(y)$. 

Note that if $f:[n]\to [l]$ is not injective, we can factor $f= h\circ s^i$ for some $0\leq i<n$ and $h:[n-1]\to[l]$. Hence $v_f(y)=0$ if $s_i^*\alpha=0$. It follows that $s_i^*\alpha=0$ for every $0\leq i<n$, if and only if $v_f(y)=0$ for every $y\in T[1]M_l$ and $f:[n]\to[l]$ non injective.

Using \eqref{eq:hatalpha2}, we see that for any $g:[n+1]\to[l]$, and any $y\in T[1]M_l$,
$$\sum_{i=0}^{n+1}(-1)^iv_{g\circ d^i}(y)=\big(\sum_{i=0}^{n+1}(-1)^i\alpha\circ T[1]M_\bullet(d^i)\big)\circ T[1]M_\bullet(g)(y).$$ 
So $\sum_{i=0}^{n+1}(-1)^iv_{g\circ d^i}(y)=0$ for any $g:[n+1]\to[l]$, and any $y\in T[1]M_l$ if and only if $$\sum_{i=0}^{n+1}(-1)^i\alpha\circ T[1]M_\bullet(d^i)=0\Leftrightarrow D\alpha=0.$$

%This shows that $D\alpha=0$ if and only if $\sum_{i=0}^{n+1}(-1)^iv_{g\circ d^i}$ for any $\hat\alpha(y)=\{v_f\}_{f:[n]\to [l]}

Comparison with Example~\ref{ex:DoldKan} shows that \eqref{eq:hatalpha2} defines a simplicial map if and only if $s_i^*\alpha=0$ and $D\alpha=0$, that is $\alpha$ is multiplicative.
\end{proof}

As a consequence, if $M_\bullet$ integrates an $NQ$-manifold $X$, and $$\alpha\in\Omega^k(M_n),\quad
\hat\alpha:T[1]M_\bullet\to K(\mbb{R}[k],n)$$ is a multiplicative $k$-form, then $$\on{1-Jet}(\hat\alpha):T[1,0]X\to\mbb{R}[k,n]$$ is a morphism of bi-graded $Q$-manifolds. That is, $\on{1-Jet}(\hat\alpha)$ defines a $Q$-invariant $k$-form of degree $n$ on $X$. Moreover, $\on{1-Jet}(\hat\alpha)$ is closed whenever $\alpha$ is.

\subsection{Courant algebroids}\label{sec:CAlg}
Recall \cite{Severa:2005vla,Roytenberg:2002}  that a Courant algebroid is a $NQ$-manifold $X$ carrying a $Q$-invariant symplectic form $\omega_X\in\Omega^2(X)$ of degree 2. If $M$ is the base of $X$, then $X$ is said to be a Courant algebroid over $M$.

\subsubsection{Standard Courant Algebroid}\label{sec:stdCAlg}
Let $M$ be a manifold. $T^*[2]T[1]M$ carries the canonical symplectic form $\omega$ of degree 2. The homological vector field on $T[1]M$, $Q_{deRham}$, lifts in the canonical way to define a $Q$-structure on $T^*[2]T[1]M$. Therefore $(T^*[2]T[1]M,Q,\omega)$ is a Courant algebroid, called the {\em standard Courant algebroid}.

For later reference, we describe $(T^*[2]T[1]M,Q,\omega)$ explicitly in coordinates.
 Let $x^a$ ($a=1,\dots,d=\on{dim}(M)$) be local coordinates on $M$. We take $\xi^a:=dx^a$ to be the coordinates on $T[1]M$ canonically associated to $x^a$. On the cotangent bundle $T^*[2]T[1]M$, we take $p_a,\eta_a$ to be the canonical coordinates associated to $x^a,\xi^a$. 
 
 \begin{remark}[Some Shorthand Notation]
 As a shorthand, we will denote a point with coordinates $$\big((x^1,\dots,x^d);(\xi^1,\dots,\xi^d);(p_1,\dots,p_d);(\eta_1,\dots,\eta_d)\big)\in \mbb{R}^d\times\mbb{R}[1]^d\times\mbb{R}[2]^d\times\mbb{R}[1]^d$$ by $(x^\alpha,\xi^\alpha,p_\alpha,\eta_\alpha)$.
 
 Furthermore, if $\{\tau_{a_1\dots a_k}\}_{\{1\leq a_1,\dots ,a_k\leq d\}}\subset C^\infty(M)$, we will often use $$\tau_{[a_1\dots a_k]}:=\frac{1}{k!}\sum_{\sigma\in\Sigma_k}(-1)^{\on{sign}(\sigma)}\tau_{\sigma(a_1)\dots\sigma(a_k)},$$ as another shorthand, where $\Sigma_k$ denotes the permutation group of $\{1,\dots,k\}$ and $\on{sign}:\Sigma_k\to \mbb{Z}_2$ is the unique non-trivial group morphism.
 
 Finally, we will also invoke Einstein's summation convention, summing over repeated upper and lower indices.
 \end{remark}
 
 Using Einstein's summation convention, the symplectic form can be written \begin{equation}\label{eq:omega}\omega:=dp_adx^a +d\eta_ad\xi^a,\end{equation}
 and the homological vector field is \begin{equation}\label{eq:Q}Q:=\xi^a\partial_{x^a}+p_a\partial_{\eta_a}=X_{\xi^a p_a}.\end{equation} The corresponding Poisson bracket is defined on coordinates by
 $$\{p_a,x^b\}=\delta_a^b,\quad \{\eta_a,\xi^b\}=\delta_a^b,\quad\text{ where }\delta_a^b:=\left\{\begin{array}{cl}
1&\text{ if }a=b\\
0&\text{ otherwise.}\end{array}\right.$$ (and the Poisson bracket of any other pair of coordinates is equal to zero).

\subsubsection{Exact Courant Algebroids}\label{sec:ExCAlg}

Exact Courant algebroids were first introduced in \cite{Severa:2001}. We recall the construction described in \cite{Roytenberg02,KosmannSchwarzbach:2005wc}. 
Let $\kappa\in\Omega^3_{cl}(M)$ be a closed 3-form.  Under the identification $C^\infty(T[1]M)\cong \Omega^\bullet(M)$, we may view $\kappa$ as a degree 3 function on $T[1]M$. 
 If $q:T^*[2]T[1]M\to T[1]M$ is the projection, then $q^*\kappa$ satisfies the Maurer-Cartan equation
$$Qq^*\kappa+\frac{1}{2}\{q^*\kappa,q^*\kappa\}=0$$ (with both terms vanishing identically).

Therefore, with $X_\kappa:=\{q^*\kappa,\cdot\}$, $$Q_\kappa:=Q+X_\kappa$$ defines a new homological vector field on $T^*[2]T[1]M$, preserving the symplectic structure. The triple $(T^*[2]T[1]M,Q_\kappa,\omega)$ is a Courant algebroid, called the {\em exact Courant algebroid with background 3-form $\kappa$.}

\subsubsection{Gauge Transformations}\label{sec:GTinf}
Suppose that $\beta\in\Omega^2(M)$ is a 2-form. The time-1 Hamiltonian flow along $X_\beta:=\{q^*\beta,\cdot\}$ preserves the symplectic structure, but deforms the homological vector field $Q_\kappa$ to 
$$e^{\on{ad}(X_\beta)}(Q_\kappa)= (1+[X_\beta,]+\frac{1}{2}[X_\beta,[X_\beta,\cdot]]+\cdots)(Q+X_\kappa)
=Q+X_{\kappa-d\beta}.$$
Consequently, the time-1 Hamiltonian flow $e^{X_\beta}$ defines an isomorphism of Courant algebroids \begin{equation}\label{eq:betatwist}e^{X_\beta}:(T^*[2]T[1]M,Q_\kappa,\omega) \to (T^*[2]T[1]M,Q_{\kappa-d\beta},\omega).\end{equation}

In particular, if $\beta$ is closed,  $e^{X_\beta}$ defines an automorphism of Courant algebroids.
Furthermore, \eqref{eq:betatwist} shows that the isomorphism class of $(T^*[2]T[1]M,Q_\kappa,\omega)$ depends only on the cohomology class of $\kappa$. Any Courant algebroid isomorphic to $(T^*[2]T[1]M,Q_\kappa,\omega)$ is called an {\em exact Courant algebroid with characteristic class $[\kappa]\in H^3(M)$}.

%\begin{remark} TODO: DO WE NEED THIS TERMINOLOGY?
%Since there is a natural inclusion $T[1]M\to T^*[2]T[1]M$ (as the zero section), we refer to the triple $(T^*[2]T[1]M,Q_\kappa,\omega)$ as a {\em split}-exact Courant algebroid. Notice that the isomorphism \eqref{eq:betatwist} does not preserve this inclusion.
%\end{remark}

\begin{definition}\label{def:sym2grp}
A {\em strictly-2-symplectic local Lie 2-groupoid} over a manifold $M$ is a pair $([\hat X_\bullet],\omega_{\hat X})$, where $[{\hat X}_\bullet]$ is a local Lie 2-groupoid with $[{\hat X}_0]=M$ and $\omega_{\hat X}\in\Omega^2([{\hat X}_2])$ is a multiplicative symplectic form.

 %The pair $(M_\bullet,\omega)$ is called a {\em local} (strong) 2-symplectic 2-groupoid if $M_\bullet$ is only a local 2-groupoid.

$([{\hat X}_\bullet],\omega_{\hat X})$ {\em integrates} the Courant algebroid $(X,\omega_X)$ if \begin{equation}\label{eq:diff}\on{1-Jet}([{\hat X}_\bullet])=X,\text{ and }\on{1-Jet}(\hat\omega_{\hat X})=\omega_X.\end{equation}
\end{definition}

\begin{remark}\label{rem:nonUniqueSymplForm}Suppose $([\hat X^1_\bullet],\omega_{\hat X^1})$ and $([\hat X^2_\bullet],\omega_{\hat X^2})$ are two integrations of a Courant algebroid $(X,\omega_X)$. 
Since $\on{1-Jet}$ is not faithful, we cannot expect a unique diffeomorphism $\phi:[\hat X^1_\bullet]\to [\hat X^2_\bullet]$ to exist such that the following diagram commutes
\begin{equation}\label{eq:DoesNotCom}\xymatrix{
[ T[1,0]\hat X^1_\bullet]\ar[r]^{d\phi}\ar[d]^{\hat\omega_{\hat X^1}}&[T[1,0]\hat X^2_\bullet]\ar[d]^{\hat\omega_{\hat X^1}}\\
K(\mbb{R}[2],2)\ar[r]&K(\mbb{R}[2],2)
}\end{equation}
That is, it is possible that no symplectomorphism $$\phi:[\hat X^1_\bullet]\to [\hat X^2_\bullet]$$ satisfying $\on{1-Jet}(\phi)=\on{id}$ exists; and if it does exist, it need not be unique.
In particular, the symplectic form $\omega_{\hat X}$ on $[\hat X_2]$ is not unique, in general. 
%IS IT CLEAR THAT THERE ARE NON-ISOMORPHIC INTEGRATIONS?
 %IS THERE A WEAKER NOTION OF ISOMORPHISM SUCH THAT ANY 2 INTEGRATIONS ARE LOCALLY ISOMORPHIC? IS THE CONJECTURE THAT ANY LIE $n$-ALGEBROID HAS A UNIQUE (UP TO A NON-UNIQUE ISO) LOCAL INTEGRATION WRONG?
%Indeed, \eqref{eq:diff} may hold even when $\omega_{\hat X}$ is degenerate \cite{Mehta:2010ue}.

For example, in \S~\ref{sec:IntStdC}, we will construct a strictly-2-symplectic (local) 2-groupoid $$([\hat X^1_\bullet],\omega_{\hat X^1}):=([\mc{T}^*_\bullet M],\omega_\mc{T})$$ integrating $(X,\omega_X):=(T^*[2]T[1]M,\omega)$,  the standard Courant algebroid. Meanwhile, Mehta and Tang \cite{Mehta:2010ue} have constructed a Lie 2-groupoid $\hat X^{2}_\bullet$ together with a closed multiplicative 2-form $\omega_{\hat X^2}\in\Omega^2({\hat X^2}_\bullet)$ such that \eqref{eq:diff} also holds. Moreover, as we shall explain in Remark~\ref{rem:isom}, there is a canonical map $\phi:[\hat X^1_\bullet]\to [\hat X^2_\bullet]$ such that $\on{1-Jet}(\phi)=\on{id}$.

However, the 2-form, $\omega_{\hat X^2}$, they construct is degenerate, even along ${\hat X^2}_0\subset {\hat X^2}_2$. Hence, there exists no map $\phi:[\hat X^1_\bullet]\to [\hat X^2_\bullet]$ such that \eqref{eq:DoesNotCom} commutes. 

This suggests that a more flexible definition of (local) symplectic 2-groupoids is needed, which interprets diagrams such as \eqref{eq:DoesNotCom} in terms of appropriate 2-morphisms. We plan to address this issue in a future paper.
\end{remark}

The remainder of this paper will be devoted to constructing strictly-2-symplectic (local) 2-groupoids integrating standard and exact Courant algebroids.%, and to lifting gauge transformations by 2-forms to the 2-groupoids.

%\S~\ref{sec:stdCAlg} and \ref{sec:ExCAlg}.

\begin{remark}\label{rem:simGlbInt}
If we ignore the symplectic structure, there is a very simple construction of a simplicial manifold integrating the standard and exact Courant algebroids. However, integrating the symplectic form is slightly more challenging, and the bulk of the work in this note will be spent doing this.

There is a natural isomorphism of $NQ$-manifolds $T^*[2]T[1]M\cong T[1]T^*[1]M$ (see \S~\ref{sec:exCAlg}). In general, if $V\to M$, $V=V^{-1}\oplus\dots \oplus V^{-m}$, is a graded vector bundle concentrated in negative degrees then it is easy to describe an integration $I_\bullet$ of $T[1]V$ such that $I_\bullet$ is a Lie $m+1$-groupoid. In the special case $V=T^*[1]M$ the simplicial manifold $I_\bullet$ is (locally) isomorphic to the simplicial manifold $\mc{T}^*_\bullet M$ we will construct in \S~\ref{sec:IntStdC}.%; this will explain our somewhat unmotivated construction of $\mc{T}^*_\bullet M$.

The simplicial manifold $I_\bullet$ is defined as follows: a point in $I_k$ is a $k+1$-tuple $(x_0,\dots,x_k)$ of points of $M$ together with a choice of $v_\phi\in V^{-l}_{x_{\phi(0)}}$ for every monotone map $\phi:[l]\to[k]$ (for every $l$), such that $v_\phi=0$ for every non-injective $\phi$.

As explained by Mehta and Tang \cite{Mehta:2010ue}, this construction has the following more conceptual explanation: $V$ can be viewed as a bundle of non-negatively graded chain complexes (with trivial differential). Using the Dold-Kan correspondence pointwise (cf. Example~\ref{ex:DoldKan}), we get a bundle $\Gamma(V)_\bullet$ of simplicial vector spaces over $M$. Next we note that the bisimplicial manifold $E_\bullet\Gamma(V)_\bullet$ integrates the bi-$NQ$-manifold $T[1,0]V$. Finally, one obtains the simplical manifold $I_\bullet:=\overline{W}\big(E_\bullet\Gamma(V)_\bullet\big)$ by applying the \emph{bar construction} functor
$\overline{W}:\on{Man}^{{\mbf{\Delta}^2}^{\on{op}}}\to\on{Man}^{{\mbf{\Delta}}^{\on{op}}}$ described in \cite{Artin:1966we,CEGARRA:2005eo,Mehta:2010ue}. 

The fact that $I_\bullet:=\overline{W}(E_\bullet\Gamma(V)_\bullet)$ integrates $T[1]V$ follows directly from the fact that $\on{1-Jet}$ intertwines the functor $\overline{W}$ with the `total-grading' functor from bi-$NQ$-manifolds to $NQ$-manifolds induced by the diagonal inclusion $\on{\underline{End}}(\mbb{R}^{0\mid 1})\subset \on{\underline{End}}(\mbb{R}^{0\mid 1})^2$.
\end{remark}

\section{Integration of the Standard Courant Algebroid}\label{sec:IntStdC}
Let $M$ be a manifold. In this section we construct a local strictly-2-symplectic 2-groupoid integrating $T^*[2]T[1]M$, the standard Courant algebroid over $M$. To do this, we first choose a connection on $M$. Let $U\subset M\times M$ be a symmetric neighbourhood of the diagonal such that for any $(x,y)\in U$ there exists a unique geodesic $$\gamma_{x,y}:[0,1]\to M, \text{ such that }\gamma_{x,y}(0)=x,\gamma_{x,y}(1)=y,\text{ and }(x,\gamma_{x,y}(t))\in U.$$
%TODO: A BIT MORE PRECISE? SHOULD $(x,\gamma_{x,y}(t))\in U$?

 For $t\in[0,1]$ let $q_t:U\to M$ be given by $q_t(x,y)=\gamma_{x,y}(t)$.  We define $W:=q_{1/2}^*T^*M$. Notice that, using the parallel transport, we can identify $W$ with $W_t:=q_{t}^*T^*M$ for any $t\in[0,1]$.

Let $U_2\subset M^3$ be given by $U_2=\{(x,y,z)|(x,y),(x,z), (y,z)\in U\}$. We have a map $m:U_2\to M^3$ given by 
$$m:(x,y,z)\mapsto \bigl(q_{1/2}(x,y),q_{1/2}(y,z),q_{1/2}(z,x)\bigr).$$ We shall suppose that $U$ is such that $m$ is an injective local diffeomorphism.

Similarly, we let $U_n\subset M^{n+1}=E_n M$ be given by $U_n=\{(z_0,\dots, z_n)\mid (z_i, z_j)\in U,\quad \forall 0\leq i,j\leq n\}$. It is clear that $U_\bullet\subset E_\bullet M$ is a simplicial submanifold.

%\begin{remark}
%A slightly more intrinsic description of $W$ can be given as follows. Suppose that we have the exact sequence $$0\to T^*M\to \mbb{T}M\xrightarrow{a} TM\to 0.$$ Furthermore, let $\tilde\nabla$ be any connection on $\mbb{T}M$ such that $$a(\tilde\nabla_X Y)=\nabla_X a(Y).$$
%
%The manifold $W$ can be identified with commuting squares of morphisms of vector bundles
%$$\xymatrix{
%TI\ar[d]^{\on{id}}\ar[r]^{\tilde f_w}&\mbb{T}M\ar[d]^a\\
%TI\ar[r]^{Tf_w}\ar[d]& TM\ar[d]\\
%I\ar[r]^{f_w}& M\\
%}$$ such that $\tilde\nabla_{Tf_w{\partial_t}}\tilde f_w=0$.
%\end{remark}

 \subsection{The 2-groupoid}\label{sec:2grpoid}
 We now construct a simplicial manifold $\mc{T}^*_\bullet M$   which integrates $T^*[2]T[1]M$.
 We define it by
 %$\mc{T}^*_0 M=M,$ and for $n>0$,
 \begin{multline*}\mc{T}^*_n M:=\{\big((z_0,\dots,z_n);\{w_{i,j}\}_{0\leq i, j\leq n}\big)\\ \text{ such that } (z_i,z_j)\in U, \text{ and } w_{i,j}\in W_{(z_i,z_j)}\text{ for all } 0\leq i,j\leq n\text{ and }w_{i,j}=-w_{j,i}\text{ for all }i,j\}.\end{multline*}
 For any monotone map $f:[n]\to[m]$ we define $\mc{T}^*(f):\mc{T}^*_m M\to \mc{T}^*_n M$ by
 \begin{align*}\big((z_0,\dots,z_m);\{w_{i,j}\}_{0\leq i, j\leq m}\big)\mapsto &\big((z_{f(0)},\dots,z_{f(n)});\{w'_{i,j}\}_{0\leq i, j\leq n}\big)\\ &\text{ where } w'_{i,j}:=w_{f(i),f(j)}.\end{align*} 
In the sequel, when describing an element of $\mc{T}^*_n M$, we will often omit the terms $w_{ii}\equiv 0$, since they are always $0$.
 We can visualize $\mc{T}^*_\bullet M$ graphically as in Figure~\ref{fig:T*2gE}. It is clear that $[\mc{T}^*_\bullet M]$ is a local 2-groupoid since simplices are fully determined by their 0 and 1 dimensional faces.
\begin{figure}
\def\svgwidth{\columnwidth}
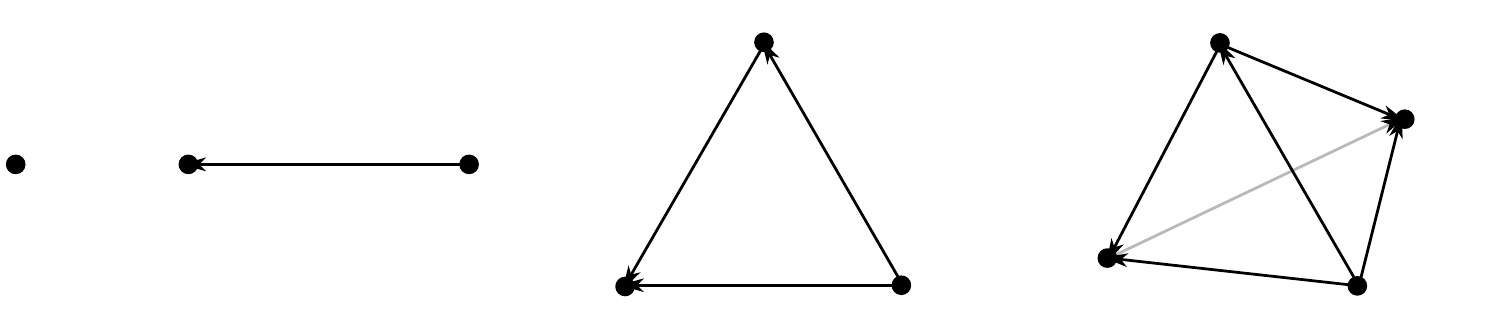
\caption{\label{fig:T*2gE} The local 2-groupoid $[\mc{T}^*_\bullet M]$. 
(Here $z_i\in M$ and $w_{i,j}\in W_{(z_i,z_j)}$.)}
\end{figure}

Recall the simplicial manifold $E_\bullet M$ defined in Example~\ref{ex:EM}.
We remark that $\mc{T}^*_\bullet M$ is a simplicial vector bundle over $U_\bullet$, an open  simplicial submanifold of $E_\bullet M$, with the projection \begin{equation}\label{eq:qmct}q_{\mc{T}}:\big((z_0,\dots,z_n);\{w_{i,j}\}_{0\leq i, j\leq n}\big)\to (z_0,\dots,z_n).\end{equation}

\begin{remark}\label{rem:isom}
If one replaces the vector bundle $W:=W_{\frac{1}{2}}$ with $ T^*M\times M\supseteq W_0$, the construction of $\mc{T}^*_\bullet M$ described above is formally identical to the construction of the simplicial manifold $I_\bullet$ described in Remark~\ref{rem:simGlbInt}.
Since parallel transport defines an isomorphism $W\cong W_0$, we get a canonical embedding of simplicial manifolds $\mc{T}^*_\bullet M\subseteq I_\bullet$. Note that $I_\bullet$ is precisely the Lie 2-groupoid described by Mehta and Tang \cite{Mehta:2010ue}.
%The simplicial manifold $\mc{T}^*_\bullet M$ canonical embedds into the simplicial manifold $I_\bullet$ described in Remark~\ref{rem:simGlbInt} via the canonical isomorphism $W=W_{\frac{1}{2}}\cong W_{0}$ defined by parallel transport.
\end{remark}
\begin{remark}
Recall that one can integrate $T^*[2]T[1]M$ to an infinite dimensional simplicial manifold $X_\bullet$ \cite{Severa:2005vla}, whose $n$-simplices are $$X_n=\on{SMan}_{NQ}\big(T[1]\lvert\Delta^n\rvert,\;T^*[2]T[1]M\big),$$ where $\lvert\Delta^n\rvert=\{(x_0,\dots,x_n)\in\mbb{R}^{n+1}\mid \sum_ix_i=1\text{ and }x_i\geq0\}$ is the geometric $n$-simplex. Using the canonical isomorphism $\omega^\flat_{T^*[1]M}\circ L: T^*[2]T[1]M\to T[1]T^*[1]M$ described in \eqref{eq:Lom}, together with Lemma~\ref{lem:Qbij}, one sees that $$X_n\cong\on{SMan}_N\big(T[1]\lvert\Delta^n\rvert,\;T^*[1]M\big)\cong \on{VBund}\big(T\lvert\Delta^n\rvert,\;T^*M\big),$$ where $\on{VBund}$ is the category of vector bundles. 

We may view our simplicial manifold $\mc{T}^*_\bullet M$ as a simplicial submanifold of \linebreak[4]
 $\on{VBund}(T\lvert\Delta^\bullet\rvert,T^*M)$. Indeed, $$\mc{T}^*_0 M=\on{VBund}(T\lvert\Delta^0\rvert,T^*M)=M.$$ 
 Meanwhile,  we can realize elements of $\mc{T}^*_n M$ ($n>0$) inductively as harmonic maps\footnote{Note that since the fibres of vector bundles are totally geodesic and linear maps are totally geodesic, a vector bundle morphism $f:T\lvert\Delta^n\rvert\to T^*M$ is harmonic if and only if its restriction $f\circ\sigma:\lvert\Delta^n\rvert\to T^*M$ to any flat section $\sigma:\lvert\Delta^n\rvert\to T\lvert\Delta^n\rvert$ is harmonic.} in $\on{VBund}(T\lvert\Delta^n\rvert,T^*M)$ which satisfy the following boundary conditions
 \begin{itemize}
 \item they restrict to elements of $\mc{T}^*_{n-1} M\subset \on{VBund}(T\lvert\Delta^{n-1}\rvert,T^*M)$ along the boundary $T\partial\lvert \Delta^n\rvert\subset T\lvert\Delta^n\rvert$, and
 \item they map horizontal vectors at points in $\big(T\partial \lvert\Delta^n\rvert\big)^\perp\subset T\lvert\Delta^n\rvert$ to horizontal vectors (horizontal with respect to the Levi-Cevita connections).
 \end{itemize}
% Additionally, we require elements in $\mc{T}^*_n M$ to be suitably 
 
 %our 
 
% 1-simplices $\mc{T}^*_1M$ are  harmonic maps $\gamma:[0,1]\cong\lvert \Delta^1\rvert\to T^*M$ (whose end points are suitably close), with the boundary condition that $\gamma$ push $\partial_t\rvert_0$ forward to a horizontal vector. %which are also horizontal with respect to the Levi-Cevita connection
 %Meanwhile, for $n>1$ we can realize $n$ simplices inductively  as (small) harmonic maps with boundary conditions determined by $n-1$ simplices. 
 
 In the construction of $\mc{T}^*_\bullet M$, we only considered `harmonic' simplicies in $X_\bullet$ which are suitably close to being degenerate. However, 
 given any element of $\on{VBund}\big(T\lvert\Delta^n\rvert,\;T^*M\big)$, one could first deform it to a homotopic `energy minimizing' map, and then look at nearby `harmonic' simplices. Patching these neighbourhoods together, one might hope to construct a global integration of $T^*[2]T[1]M$.
%Our simplicial manifold $\mc{T}^*_n M$ is a sub simplicial manifold of $\on{VBund}\big(T\lvert\Delta^n\rvert,\;T^*M\big),$ cut out by considering harmonic maps $\lvert\Delta^n\rvert\to T^*M$, and 

%Since $X_n$ is very big, one could instead consider the subset $\on{VBund}_{harm}\big(T\lvert\Delta^n\rvert,\;T^*M\big)$ of harmonic maps.
%Our simplicial manifold $\mc{T}^*_n M$ is a small open neighbourhood of the degenerate simplices in $\on{VBund}_{harm}\big(T\lvert\Delta^n\rvert,\;T^*M\big)$.
\end{remark}

\subsection{The 2-form}
Next we construct a multiplicative symplectic form on $\mc{T}^*_2 M$ which integrates the symplectic structure on $T^*[2]T[1]M$.

The map $m:U_2\to M^3$ lifts to a local diffeomorphism $\tilde m:\mc{T}^*_2 M\to T^*M^3$,
$$\tilde m: \bigl( (x_0,x_1,x_2); (w_{01}, w_{02},w_{12}) \bigr) \mapsto \bigl( m(x_0,x_1,x_2);(w_{01}, w_{12},w_{20})\bigr).  $$
 The symplectic form $\omega_\mc{T}$ on $\mc{T}^*_2 M$ is defined to be the pullback of the symplectic form on $T^*M^3$.

To check that $\omega_\mc{T}$ is multiplicative, i.e.\ that $D\omega=0$, it suffices to show that $\omega_\mc{T}=D\alpha$ for some $\alpha\in\Omega^2(\mc{T}^*_1 M)$. Let $\tilde q_{1/2}:\mc{T}^*_1 M\to T^*M$ denote the map defined by $(x,y;w)\mapsto (q_{1/2}(x,y),w)$.
We can use $\alpha=\tilde q_{1/2}^*\omega_{T^*M}$, where $\omega_{T^*M}$ is the symplectic form on $T^*M$.

\subsection{Proof that $([\mc{T}^*_\bullet M],\omega_{\mc{T}})$ integrates $T^*[2]T[1]M$.}
\begin{theorem}
The strictly-2-symplectic local 2-groupoid $([\mc{T}^*_\bullet M],\omega_{\mc{T}})$ integrates the standard Courant algebroid $(T^*[2]T[1]M,Q,\omega)$.
\end{theorem}

\begin{proof}
First we calculate $\on{1-Jet}([\mc{T}^*_\bullet M])$. Since \eqref{eq:qmct} defines a simplicial map $q_{\mc{T}}:\mc{T}^*_\bullet M\to E_\bullet M$, it induces a map
$$q_{\mc{T}}:\on{SMan}^{{\mbf{\Delta}}^{\on{op}}}(X\times E_\bullet\mbb{R}^{0\mid 1},\mc{T}^*_\bullet M)\to \on{SMan}^{{\mbf{\Delta}}^{\on{op}}}(X\times E_\bullet\mbb{R}^{0\mid 1},E_\bullet M),$$ which is natural in $X$, and therefore a map $$\tilde q_{\mc{T}}:\on{1-Jet}([\mc{T}^*_\bullet M])\to \on{1-Jet}([E_\bullet M])=T[1]M,$$
where the identity $\on{1-Jet}([E_\bullet M])=T[1]M$ was described in Example~\ref{ex:EMintTM}.

For $v\in T[1]_xM$, the corresponding map $\on{ev}_v:E_1\mbb{R}^{0\mid 1}\to E_1M$ is given by $$\on{ev}_v:(\theta_0,\theta_1)\to(x+\theta_0v,x+\theta_1v).$$ We would like to look at all possible maps $E_\bullet\mbb{R}^{0\mid 1}\to \mc{T}^*_\bullet M$ lying over the map $\on{ev}_v$.

Now a map $F_v:E_1\mbb{R}^{0\mid 1}\to \mc{T}^*_1 M$ such that $\tilde q_{\mc{T}}(F_v)=\on{ev}_v$ must be of the form
\begin{equation}\label{eq:Fv}F_v:(\theta_0,\theta_1)\to\big((x+\theta_0v,x+\theta_1v);(w_{0,1}=a\theta_0+b\theta_1+c\theta_0\theta_1)\big),\end{equation}
for $a,b\in W_{x+\theta_0v,x+\theta_1v}[1]\cong T^*[1]_xM$ and $c\in W_{x+\theta_0v,x+\theta_1v}[2]\cong T^*[2]_xM$. 

If we would like to extend $F_v:E_1\mbb{R}^{0\mid 1}\to \mc{T}^*_1 M$ to a map of simplicial (super)-manifolds, the extension must be equivariant with respect to the action of the category ${\mbf{\Delta}}^{\on{op}}$.
%compatible with all the maps $f:[m]\to[n]$. 
In particular, we must have
\begin{equation}\label{eq:obsToExt}F_v\circ s_0=s_0\circ \on{ev}_v\Rightarrow a=-b.\end{equation}
%However, if we require $F_v\circ s_0=s_0\circ \on{ev}_v$, we must have $a=-b$. 
%Assuming \eqref{eq:obsToExt} to hold, we 
However \eqref{eq:obsToExt} is the only obstruction to extending $F_v$. A routine check shows that $$(\theta_0,\dots,\theta_n)\to \big((x+\theta_0v,\dots,x+\theta_nv);(w_{ij}=b(\theta_j-\theta_i)+c\theta_i\theta_j)\big)$$
 is the unique extension of \eqref{eq:Fv} to a map $E_\bullet\mbb{R}^{0\mid 1}\to \mc{T}^*_\bullet M$.

Consequently simplicial maps $E_\bullet \mbb{R}^{0\mid 1}\to \mc{T}^*_\bullet M$ are parametrized by the super-manifold $T^*[2]T[1]M$. In the coordinates described in \S~\ref{sec:CAlg}, the maps are 
$$\on{ev}_{(x^\alpha,\xi^\alpha,p_\alpha,\eta_\alpha)}:(\theta_0,\dots,\theta_n)\to \big((x^\alpha+\theta_0\xi^\alpha,\dots,x^\alpha+\theta_n\xi^\alpha);(w_{ij}=\eta_\alpha(\theta_j-\theta_i)+p_\alpha\theta_i\theta_j)\big),$$ see Figure~\ref{fig:Maps}. It follows that $\on{1-Jet}([\mc{T}^*_\bullet M])=T^*[2]T[1]M$ as a super-manifold.
\begin{figure}
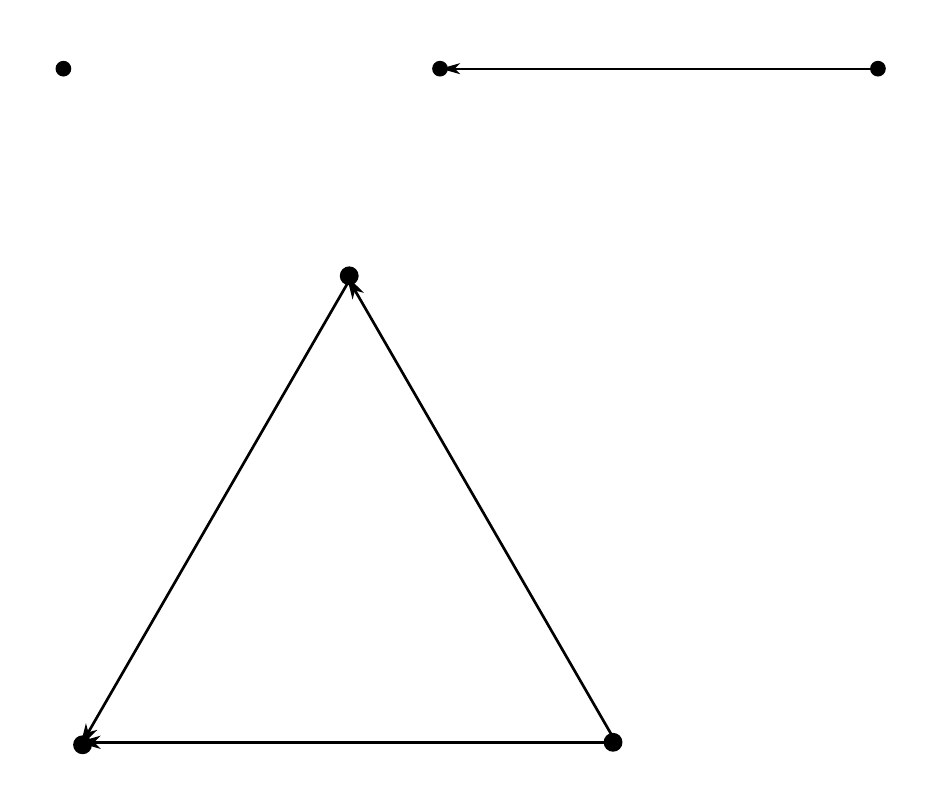
\caption{\label{fig:Maps} The Maps $E_\bullet\mbb{R}^{0\mid 1}\to \mc{T}^*_\bullet M$ parametrized by $T^*[2]T[1]M$.}
\end{figure}

To show that $\on{1-Jet}([\mc{T}^*_\bullet M])$ is the $NQ$-manifold $T^*[2]T[1]M$, we must show that the $Q$ structure on $\on{1-Jet}([\mc{T}^*_\bullet M])$ is given by \eqref{eq:Q}.
The natural action of $\mbb{R}^{0\mid 1}$ on $E_\bullet \mbb{R}^{0\mid 1}$ is $$\theta\cdot(\theta_0,\dots,\theta_n)=(\theta_0+\theta,\dots,\theta_n+\theta).$$ So we see that 
\begin{multline*}\on{ev}_{(x^\alpha,\xi^\alpha,p_\alpha,\eta_\alpha)}:\theta\cdot(\theta_0,\dots,\theta_n)\to \big((x^\alpha+\theta\xi^\alpha+\theta_0\xi^\alpha,\dots,x^\alpha+\theta\xi^\alpha+\theta_n\xi^\alpha);\\(w_{ij}=(\eta_\alpha+\theta p_\alpha)(\theta_j-\theta_i)+p_\alpha\theta_i\theta_j)\big),\end{multline*} which descends to the action \begin{equation}\label{eq:Qflow}\theta\cdot(x^\alpha,\xi^\alpha,p_\alpha,\eta_\alpha)=(x^\alpha+\theta\xi^\alpha,\xi^\alpha,p_\alpha,\eta_\alpha+\theta p_\alpha)\end{equation} on $\on{1-Jet}([\mc{T}^*_\bullet M])=T^*[2]T[1]M$. Since \eqref{eq:Qflow} is just the flow along the vector field \eqref{eq:Q}, it follows that $\on{1-Jet}([\mc{T}^*_\bullet M])=T^*[2]T[1]M$ as an $NQ$-manifold. 

What remains is to check that $\omega_{\mc{T}}\in\Omega^2(\mc{T}^*_2 M)$ integrates the canonical symplectic form on the cotangent bundle $T^*[2]T[1]M$. For this purpose we need to find the pullback of $\omega_\mc{T}$ by the evaluation map $\on{ev}_2:T^*[2]T[1]M\times E_2\mbb{R}^{0\mid 1}\to\mc{T}^*_2 M$. We have 
$$q_{1/2}(x^\alpha + \theta_i\xi^\alpha, x^\alpha + \theta_j\xi^\alpha)=x^\alpha + \frac{\theta_i+\theta_j}{2}\xi^\alpha.$$
As a result, 
$$\on{ev}_2^*\omega_\mc{T}=\sum_{(i,j)=(0,1),(1,2),(2,0)}d\bigl(x^\alpha + \frac{\theta_i+\theta_j}{2} \xi^\alpha\bigr)\, d \bigl(\eta_\alpha(\theta_j-\theta_i) + p_\alpha\theta_i\theta_j\bigr).$$
The $(2,0)$-bihomogeneous part of $\on{ev}_2^*\omega_\mc{T}$ (the part which is a 2-form on $T^*[2]T[1]M$ para\-metrized by $E_2\mbb{R}^{0\mid 1}$) is equal to
$$(\on{ev}_2^*\omega_\mc{T})^{(2,0)}=(dp_\alpha\,dx^\alpha + d\eta_\alpha\,d\xi^\alpha)(\theta_0\theta_1+\theta_1\theta_2+\theta_2\theta_0).$$
Comparison with \eqref{eq:VEKVmap} shows that $\on{1-Jet}(\hat\omega_\mc{T})=d p_\alpha dx^\alpha+d\eta_\alpha d\xi^\alpha$. Since the symplectic form on $T^*[2]T[1]M$ is given by \eqref{eq:omega}, this concludes the proof.
%both of which are the map $E_2\mbb{R}^{0\mid 1}\to K(\mbb{R}[1],2)_2$ defined by the element $

\end{proof}

\begin{remark}
In \cite{Mehta:2010ue}, Rajan Mehta and Xiang Tang construct a (global) Lie 2-groupoid by applying the bar construction to the bi-simplicial manifold integrating the Lie bialgebroid $(T^*M,TM)$. As explained in Remark~\ref{rem:isom}, $[\mc{T}^*_\bullet M]$ is isomorphic to the local version of the Lie 2-groupoid they build.

Furthermore, they construct a closed (degenerate) two-form $\omega'_\mc{T}$ on the 2-simplices of $\mc{T}^*_2 M$.
A calculation similar to the proof above shows that $\on{1-Jet}(\omega'_\mc{T})=\omega_{T^*[2]T[1]M}$, therefore, their construction differentiates to the standard Courant algebroid.
%It follows from the proof above that $\on{1-Jet}(\lambda_0)=\omega$.

%The advantage of our construction is that $\omega_\mc{T}$ is symplectic. 
%The use of the bar functor provides a nice conceptual underpinning to the construction in \cite{Mehta:2010ue}. Unfortunately, as a result, their construction depends on a decomposition of the Courant algebroid into a bialgebroid. On the other hand, $\mc{T}^*_\bullet M$ depends only on the choice of a connection on $M$. 

The requirement in Definition~\ref{def:sym2grp} that the 2-form be non-degenerate ``on the nose'',  is (at least at first glance) less complicated than the corresponding non-degeneracy requirements in \cite[\S~6.2]{Mehta:2010ue}. However, it would be interesting to know if Mehta and Tang's requirements on their 2-form are equivalent to the $\on{1-Jet}$ of their 2-form being non-degenerate ``on the nose''. Additionally, our 2-form is only defined on a local Lie 2-groupoid, while theirs is defined on a genuine Lie 2-groupoid.

Notice that, unlike the construction in \cite{Mehta:2010ue}, $\mc{T}^*_\bullet M:\mbf{\Delta}^{\on{op}}\to \on{Man}$ extends naturally to a functor $\text{FSets}^\text{op}\to \on{Man}$ (where $\text{FSets}$ is the category of finite sets). 
\end{remark}

\section{Integration of Exact Courant Algebroids}\label{sec:exCAlg}
In this section, we show that as $NQ$-manifolds, any exact Courant algebroid over $M$ is isomorphic to the standard Courant algebroid over $M$ (see also \cite{Sheng:2011vz}). Therefore {\em once we modify the symplectic structure}, the construction in \S~\ref{sec:IntStdC} also integrates exact Courant algebroids.

\begin{lemma}\label{lem:Qbij}
Let $F:\on{SMan}_{NQ}\to\on{SMan}_N$ be the forgetful functor. Then $T[1]:\on{SMan}_N\to \on{SMan}_{NQ}$ is a right adjoint for $F$. In particular, for $X\in \on{SMan}_{NQ}$ and $Z\in\on{SMan}_N$, $$\on{SMan}_{NQ}(X,T[1]Z)\cong\on{SMan}_N(X,Z).$$
%If $X$ is any $NQ$-manifold, let 
%and $Z$ is any $N$-manifold, then $NQ$-morphisms $$X\to T[1]Z$$ are in one-to-one correspondence with $N$-morphisms $$X\to Z.$$ 
\end{lemma}

\begin{proof}
Since $T[1]Z$ represents the pre-sheaf $\on{SMan}(\cdot\times \mbb{R}^{0\mid 1},Z)$, we have $$\on{SMan}_{NQ}(X,T[1]Z)=\on{SMan}_N(X\times_{\mbb{R}^{0\mid 1}}\mbb{R}^{0\mid 1},Z),$$ but the action of $\mbb{R}^{0\mid 1}$ on $X$ defines an isomorphism $X\times_{\mbb{R}^{0\mid 1}}\mbb{R}^{0\mid 1}\cong X$.

It is clear that this bijection $\on{SMan}_{NQ}(X,T[1]Z)\cong\on{SMan}_N(X,Z)$ is natural in both $X$ and $Z$, which concludes the proof.
%\end{proof}

%\begin{remark}\label{rem:Qbij}
Let us describe the unit and co-unit of the adjunction explicitly. Note that the homological vector field $Q_X$ on $X$ defines a section $Q_X:X\to T[1]X$. Furthermore, $Q_X:X\to T[1]X$ is a $NQ$-morphism with respect to the de Rham vector field on $T[1]X$. The unit of the adjunction is the natural transformation $$1\to T[1]\circ F,\quad X\to (Q_X:X\to T[1]X).$$ 

Meanwhile, if $q_{Z}:T[1]Z\to Z$ is the canonical projection, then the co-unit is the natural transformation $$F\circ T[1]\to 1,\quad Z\to (q_Z:T[1]Z\to Z).$$

In particular, if $f:X\to Z$ is any $N$-morphism, $$T[1]f\circ Q_X:X\to T[1]Z$$ is the corresponding $NQ$-morphism.
In the other direction, if $\tilde f:X\to T[1]Z$ is an $NQ$-morphism, the corresponding $N$-morphism is just $q_Z\circ\tilde f$.
%\end{remark}
\end{proof}

\begin{example}[The Anchor Map]
Suppose $X$ is any $NQ$-manifold over the base $B(X):=(0,0)\cdot X$ (here $(0,0)\in\on{\underline{End}}(\mbb{R}^{0\mid 1})$, see Definition~\ref{def:NQmfld}). Multiplication by $(0,0)$ defines the canonical projection $q_X:X\to B(X)$. It follows from Lemma~\ref{lem:Qbij} that there is a unique $NQ$-morphism $$a_X:X\to T[1]B(X)$$ over $q_X:X\to B(X)$, called the {\em anchor map} for $X$. 

Furthermore, Lemma~\ref{lem:Qbij}  implies that $a_X$ is a natural transformation from the identity functor on $\on{SMan}_{NQ}$ to $T[1]\circ B:\on{SMan}_{NQ}\to\on{SMan}_{NQ}$. In fact $T[1]:\on{Man}\to\on{SMan}_{NQ}$ is a right adjoint for $B:\on{SMan}_{NQ}\to \on{Man}$, with unit $a_X$ and co-unit the identity.
%Lemma~\ref{lem:Qbij} and Remark~\ref{rem:Qbij} also imply that if $f:X\to Y$ is any $NQ$-morphism over a map $f_0:X_0\to Y_0$ of their base manifolds, then we have the following commutative diagram of $NQ$-morphisms
%$$\xymatrix{
%X\ar[r]^f\ar[d]^{a_X}&X\ar[d]^{a_Y}\\T[1]X_0\ar[r]^{T[1]f_0}&T[1]Y_0
%}$$
\end{example}

Let \begin{subequations}\label{eq:Lom}\begin{equation}L:T^*[2]T[1]M\to T^*[2]T^*[1]M\end{equation} denote the Legendre transform and \begin{equation}\omega_{T^*[1]M}^\flat:T^*[2]T^*[1]M\to T[1]T^*[1]M\end{equation}\end{subequations} be the diffeomorphism defined by the canonical symplectic structure $\omega_{T^*[1]M}$ on $T^*[1]M$ \cite{Mackenzie-Xu94,Roytenberg:2002}. Note that the composition $\omega_{T^*[1]M}^\flat\circ L$ is an $NQ$-morphism with respect to the de Rham vector field on $T[1]T^*[1]M$ \cite{Roytenberg:2002}. Finally, let $s:T[1]T^*[1]M\to T^*[1]M$ denote the canonical projection.

Suppose that $(T^*[2]T[1]M,Q_\kappa,\omega)$ is a exact Courant algebroid with background 3-form $\kappa\in\Omega^3(M)$. Then by Lemma~\ref{lem:Qbij}, there is a $NQ$-morphism $\psi_\kappa:(T^*[2]T[1]M,Q_\kappa)\to (T^*[2]T[1]M,Q)$ corresponding to the $N$-morphism $s\circ\omega_{T^*[1]M}^\flat\circ L$.

\begin{proposition}\label{prop:psiK}
The $NQ$-morphism $\psi_\kappa:(T^*[2]T[1]M,Q_\kappa)\to (T^*[2]T[1]M,Q)$ is given by translation along the fibres of $q:T^*[2]T[1]M\to T[1]M$ by the degree-2 1-form $$\frac{1}{2}\iota_{Q_{dR}}\iota_{Q_{dR}}q_M^*\kappa,$$ where $q_M:T[1]M\to M$ is the projection, and $Q_{dR}$ is the de Rham vector field on $T[1]M$.
In particular, $\psi_\kappa$ is a diffeomorphism.

%a diffeomorphism. Explicitly,  if $q_M:T[1]M\to M$ is the projection, and $Q_{dR}$ is the de Rham vector field on $T[1]M$, then $\psi_\kappa$ is given by translation along the fibres of $q:T^*[2]T[1]M\to T[1]M$ by the degree-2 1-form $\iota_{Q_{dR}}\iota_{Q_{dR}}q_M^*\kappa$.
\end{proposition}
\begin{proof}
A routine calculation shows that, in coordinates, both maps are equal to $$\psi_\kappa:(x^\alpha,\xi^\alpha,p_\alpha,\eta_\alpha)\to(x^\alpha,\xi^\alpha,p_\alpha+3\kappa_{[\alpha bc]}\xi^b\xi^c,\eta_\alpha),$$
where %$\beta=\frac{1}{2}\beta_{ab}dx^adx^b$ and 
$\kappa=\kappa_{[abc]}dx^adx^bdx^c$. 

%If $q_M:T[1]M\to M$ is the projection, and $Q_{deRham}$ is the de Rham vector field on $T[1]M$, then
%In other words, $\psi_\kappa$ is just the fibrewise translation of $T^*[2]T[1]M$ by the 1-form $\iota_{Q_{deRham}}\iota_{Q_{deRham}}\kappa$
\end{proof}

It follows that, under the diffeomorphism $\psi_\kappa$, the canonical symplectic form $\omega$ on $T^*[2]T[1]M$ transforms as $$(\psi_\kappa^{-1})^*\omega=\omega-\frac{1}{2}q^*d\iota_{Q_{dR}}\iota_{Q_{dR}}q_M^*\kappa=\omega-q^*\Lied_{Q_{dR}}\iota_{Q_{dR}}q_M^*\kappa.$$

Therefore, we need to modify the symplectic form on $\mc{T}^*_2 M$ by a multiplicative term which is mapped to $-\Lied_{Q_{dR}}\iota_{Q_{dR}}q_M^*\kappa$ by $\on{1-Jet}$.
%Therefore, we need to find a multiplicative 2-form $\varkappa\in \Omega^2(E_2M)$ such that $\on{1-Jet}(\varkappa)=-2\Lied_{Q_{dR}}\iota_{Q_{dR}}q_M^*\kappa$.

\subsection{Modification of the symplectic form on $\mc{T}^*_2 M$}
To lift this modification of the symplectic form to $\mc{T}^*_2 M$, we make use of the connection on $M$ to integrate $\kappa$ along geodesics in $M$.

Recall from \S~\ref{sec:2grpoid} that $\mc{T}^*_\bullet M$ is a simplicial vector bundle over the  simplicial manifold $U_\bullet\subseteq E_\bullet M$. We define a 2-form $\mu\in\Omega^2(U_1)$ 
\begin{equation}\label{eq:mu}\mu:=\fint_{[0,1]}\Gamma^*\kappa\end{equation}
by integration over the fibres of \begin{equation}\label{eq:Gamma}\Gamma:[0,1]\times U_1 \to M,\quad \Gamma(t,x,y)=\gamma_{x,y}(t).\end{equation}
Let $\varkappa\in\Omega^2(U_2)$ be defined by $$\varkappa=D\mu:=\sum_{i=0}^2 (-1)^i d_i^*\mu.$$ We see immediately that $D\varkappa=0$, so it is multiplicative. Furthermore, $$d\fint_{[0,1]}\Gamma^*\kappa-\fint_{[0,1]}d\Gamma^*\kappa=d_1^*\kappa-d_0^*\kappa.$$ Since $\kappa$ is closed, this implies that $d\mu=d_1^*\kappa-d_0^*\kappa$. Consequently $\varkappa$ is closed.

\begin{lemma}\label{lem:1jvk=lqqk}
$\on{1-Jet}(\hat\varkappa)=\Lied_{Q_{dR}}\iota_{Q_{dR}}q_M^*\kappa$
\end{lemma}
The two form $\omega_{\mc{T}}-q_\mc{T}^*\varkappa$ is non-degenerate since we are modifying the standard symplectic form on  of $T^*M^3$ by a closed 2-form on (an open subset of) $M^3$. Thus, we immediately get the following theorem.
\begin{theorem}
The strictly-2-symplectic local 2-groupoid $([\mc{T}^*_\bullet M],\omega_{\mc{T}}-q_\mc{T}^*\varkappa)$ integrates the exact Courant algebroid $(T^*[2]T[1]M,Q_\kappa,\omega)$ with background 3-form $\kappa\in\Omega^3(M)$.
\end{theorem}

\begin{proof}[proof of Lemma~\ref{lem:1jvk=lqqk}]
We do the calculation in coordinates. 
Recall from Example~\ref{ex:EMintTM} that $T[1]M$ parametrizes the simplicial maps $E_\bullet\mbb{R}^{0\mid 1}\to E_\bullet M$. A point $(x^\alpha,\xi^\alpha)\in T[1]M$ parametrizes the map $$\on{ev}_{(x^\alpha,\xi^\alpha)}:(\theta_0,\dots,\theta_n)\to (x^\alpha+\theta_0\xi^\alpha,\dots,x^\alpha+\theta_n\xi^\alpha).$$
Composition with $\Gamma:[0,1]\times E_1M\to M$ is given by
$$\Gamma:\big(t,\on{ev}_{(x^\alpha,\xi^\alpha)}(\theta_0,\theta_1)\big)\to
(x^\alpha+(t(\theta_1-\theta_0)+\theta_0)\xi^\alpha).$$

Therefore, with $\kappa=\kappa_{[abc]}dx^adx^bdx^c$,
$$\begin{array}{rl}
&\Gamma^*\kappa_{[abc]}dx^adx^bdx^c\\
=&\kappa_{[abc]}(x^\alpha+(\cdots)\xi^\alpha)d(x^a+(\cdots)\xi^a)d(x^b+(\cdots)\xi^b)d(x^c+(\cdots)\xi^c)\\
=&\big(\kappa_{[abc]}(x^\alpha)+\frac{\partial\kappa_{[abc]}}{\partial x^d}(x^\alpha)(t(\theta_1-\theta_0)+\theta_0)\xi^d+\cdots)\\
%&\big( 3(\theta_1-\theta_0)\xi^a dt dx^b dx^c+6(\theta_1-\theta_0)\xi^a (t(\theta_1-\theta_0)+\theta_0)dt d\xi^b dx^c+\cdots\big)\\
&\big( 3(\theta_1-\theta_0)\xi^a dt dx^b dx^c+6\theta_0\theta_1\xi^a dt d\xi^b dx^c+\cdots\big)\\
=&\big(\kappa_{[abc]}(3(\theta_1-\theta_0)\xi^a dx^b dx^c-6\theta_0\theta_1\xi^a d\xi^b dx^c)-3\frac{\partial\kappa_{[abc]}}{\partial x^d}\theta_0\theta_1\xi^d\xi^a dx^b dx^c\big)dt+\cdots
\end{array}$$
Where we have omitted the terms that do not contain $dt$.
So at the point $\on{ev}_{(x^\alpha,\xi^\alpha)}(\theta_0,\theta_1)$, $\mu$ is given by 
$$\mu(\on{ev}_{(x^\alpha,\xi^\alpha)}(\theta_0,\theta_1))=\kappa_{[abc]}(3(\theta_1-\theta_0)\xi^a dx^b dx^c-6\theta_0\theta_1\xi^a d\xi^b dx^c)-3\frac{\partial\kappa_{[abc]}}{\partial x^d}\theta_0\theta_1\xi^d\xi^a dx^b dx^c$$

%Recall that the map $E_2\mbb{R}^{0\mid 1}\to E_2 M$ parametrized by the point $(x^\alpha,\xi^\alpha)\in T[1]M$ is given by $$\on{ev}_{(x^\alpha,\xi^\alpha)}:(\theta_0,\theta_1,\theta_2)\to (x^\alpha+\theta_0\xi^\alpha,x^\alpha+\theta_1\xi^\alpha,x^\alpha+\theta_2\xi^\alpha),$$ s
Therefore, at the point $\on{ev}_{(x^\alpha,\xi^\alpha)}(\theta_0,\theta_1,\theta_2)$, the 2-form $\varkappa$ is given by
$$\varkappa(\on{ev}_{(x^\alpha,\xi^\alpha)}(\theta_0,\theta_1,\theta_2))=-3\big(2\kappa_{[abc]}\xi^a d\xi^b dx^c+\frac{\partial\kappa_{[abc]}}{\partial x^d}\xi^d\xi^a dx^b dx^c\big)(\theta_0\theta_1+\theta_1\theta_2+\theta_2\theta_0).$$

Comparison with \eqref{eq:VEKVmap} shows that $$\on{1-Jet}(\hat\varkappa)=-6\kappa_{[abc]}\xi^a d\xi^b dx^c-3\frac{\partial\kappa_{[abc]}}{\partial x^d}\xi^d\xi^a dx^b dx^c.$$
Using $d\kappa=0$, we calculate,
\begin{align*}
\on{1-Jet}(\hat\varkappa)%&=-6\kappa_{[abc]}\xi^a d\xi^b dx^c-3\frac{\partial\kappa_{[abc]}}{\partial x^d}\xi^d\xi^a dx^b dx^c\\
&=6\kappa_{[abc]}d\xi^a \xi^b dx^c+3\frac{\partial\kappa_{[abc]}}{\partial x^d}dx^d \xi^a \xi^b dx^c\\
&=3d(\kappa_{[abc]}\xi^a\xi^bdx^c)\\
&=\frac{1}{2}d\iota_{Q_{dR}}\iota_{Q_{dR}}q_M^*\kappa\\
&=\Lied_{Q_{dR}}\iota_{Q_{dR}}q_M^*\kappa
\end{align*}
\end{proof}

\bibliography{basicbib7}{}

\providecommand{\bysame}{\leavevmode\hbox to3em{\hrulefill}\thinspace}
\providecommand{\MR}{\relax\ifhmode\unskip\space\fi MR }
\providecommand{\MRhref}[2]{%
  \href{http://www.ams.org/mathscinet-getitem?mr=#1}{#2}
}
\providecommand{\href}[2]{#2}
\begin{thebibliography}{10}

\bibitem{abad2009representations}
Camilo Arias~Abad and Marius Crainic, \emph{{Representations up to homotopy of
  Lie algebroids}},  (2009), Available at
  \url{http://arxiv.org/pdf/0901.0319v2}.

\bibitem{arias2010a_infty}
Camilo Arias~Abad and Florian Schaetz, \emph{{The $A_\infty$ de Rham theorem
  and integration of representations up to homotopy}},  (2010), Available at
  \url{http://arxiv.org/pdf/1011.4693}.

\bibitem{Artin:1966we}
Michael Artin and Barry Mazur, \emph{{On the van Kampen theorem}}, Topology
  \textbf{5} (1966), 179--189.

\bibitem{Blohmann:2010ub}
Christian Blohmann, Marco Cezar~Barbosa Fernandes, and Alan Weinstein,
  \emph{{Groupoid symmetry and constraints in general relativity}},  (2010),
  Available at \url{http://arxiv.org/pdf/1003.2857v2}.

\bibitem{Bursztyn:2010wb}
Henrique Bursztyn and Alejandro Cabrera, \emph{{Multiplicative forms at the
  infinitesimal level}},  (2010), 1--35, Available at
  \url{http://arxiv.org/pdf/1001.0534v2}.

\bibitem{Bursztyn09}
Henrique Bursztyn, Alejandro Cabrera, and Cristi{\'a}n Ortiz, \emph{{Linear and
  multiplicative 2-forms}}, Letters in Mathematical Physics \textbf{90} (2009),
  no.~1-3, 59--83.

\bibitem{Bursztyn03-1}
Henrique Bursztyn, Marius Crainic, Alan Weinstein, and Chenchang Zhu,
  \emph{{Integration of twisted Dirac brackets}}, Duke Mathematical Journal
  \textbf{123} (2004), no.~3, 549--607.

\bibitem{SCattaneo:2004fe}
Alberto~Sergio Cattaneo, \emph{{Integration of twisted Poisson structures}},
  Journal of Geometry and Physics \textbf{49} (2004), no.~2, 187--196.

\bibitem{Cattaneo:2010uq}
Alberto~Sergio Cattaneo, Benoit Dherin, and Alan Weinstein, \emph{{Symplectic
  microgeometry I: micromorphisms}}, The Journal of Symplectic Geometry
  \textbf{8} (2010), no.~2, 205--223.

\bibitem{CEGARRA:2005eo}
Antonio~M. Cegarra and Josu{\'e} Remedios, \emph{{The relationship between the
  diagonal and the bar constructions on a bisimplicial set}}, Topology and its
  Applications \textbf{153} (2005), no.~1, 21--51.

\bibitem{Crainic:2003bt}
Marius Crainic, \emph{{Differentiable and algebroid cohomology, van Est
  isomorphisms, and characteristic classes}}, Commentarii Mathematici Helvetici
  \textbf{78} (2003), no.~4, 681--721.

\bibitem{Lie-Algebroids}
Marius Crainic and Rui~Loja Fernandes, \emph{{Integrability of Lie brackets}},
  Annals of Mathematics. Second Series \textbf{157} (2003), no.~2, 575--620.

\bibitem{Crainic02}
\bysame, \emph{{Integrability of Poisson brackets}}, Journal of Differential
  Geometry \textbf{66} (2004), no.~1, 71--137.

\bibitem{goerss2009simplicial}
Paul~G. Goerss and John~F. Jardine, \emph{{Simplicial homotopy theory}},
  Progress in Mathematics, vol. 174, Birkh\"auser Verlag, 2009.

\bibitem{gracia2010lie}
Alfonso Gracia-Saz and Rajan~Amit Mehta, \emph{{Lie algebroid structures on
  double vector bundles and representation theory of Lie algebroids}}, Advances
  in Mathematics \textbf{223} (2010), no.~4, 1236--1275.

\bibitem{GraciaSaz:2010vm}
\bysame, \emph{{VB-groupoids and representation theory of Lie groupoids}},
  (2010), 1--34.

\bibitem{Henriques:2008cda}
Andr{\'e} Henriques, \emph{{Integrating $L_\infty$-algebras}}, Compositio
  Mathematica \textbf{144} (2008), no.~4, 1017--1045.

\bibitem{Ponte:2005txa}
David Iglesias~Ponte, Camille Laurent-Gengoux, and Ping Xu, \emph{{Universal
  lifting theorem and quasi-Poisson groupoids}},  (2005), 1--46, Available at
  \url{http://arxiv.org/pdf/math/0507396v1}.

\bibitem{Kochan03}
Denis Kochan, \emph{{Differential gorms and worms}}, Mathematical physics,
  World Sci. Publ., Hackensack, NJ, 2005, pp.~128--130.

\bibitem{kontsevich03}
Maxim Kontsevich, \emph{{Deformation quantization of Poisson manifolds}},
  Letters in Mathematical Physics \textbf{66} (2003), no.~3, 157--216.

\bibitem{KosmannSchwarzbach:2005wc}
Yvette Kosmann-Schwarzbach, \emph{{Quasi, twisted, and all that$\ldots$in
  Poisson geometry and Lie algebroid theory}}, The breadth of symplectic and
  Poisson geometry, Birkh\"auser Boston, Boston, MA, 2005, pp.~363--389.

\bibitem{Mackenzie-Xu94}
Kirill Charles~Howard Mackenzie and Ping Xu, \emph{{Lie bialgebroids and
  Poisson groupoids}}, Duke Mathematical Journal \textbf{73} (1994), no.~2,
  415--452.

\bibitem{Mackenzie97}
\bysame, \emph{{Integration of Lie bialgebroids}}, Topology \textbf{39} (2000),
  no.~3, 445--467.

\bibitem{Mehta:2010ue}
Rajan~Amit Mehta and Xiang Tang, \emph{{From double Lie groupoids to local Lie
  2-groupoids}}, Bulletin of the Brazilian Mathematical Society. New Series.
  Boletim da Sociedade Brasileira de Matem\'atica \textbf{42} (2011), no.~4,
  651--681, Available at \url{http://arxiv.org/pdf/1012.4103}.

\bibitem{Milnor:1964tm}
John~Willard Milnor, \emph{{Microbundles. I}}, Topology \textbf{3} (1964),
  no.~suppl. 1, 53--80.

\bibitem{Roytenberg99}
Dmitry Roytenberg, \emph{{Courant algebroids, derived brackets and even
  symplectic supermanifolds}}, Ph.D. thesis, University of California,
  Berkeley, 1999.

\bibitem{Roytenberg:2002}
\bysame, \emph{{On the structure of graded symplectic supermanifolds and
  Courant algebroids}}, Quantization, Poisson brackets and beyond (Manchester,
  2001), Amer. Math. Soc., Providence, RI, 2002, pp.~169--185.

\bibitem{Roytenberg02}
\bysame, \emph{{Quasi-Lie bialgebroids and twisted Poisson manifolds}}, Letters
  in Mathematical Physics \textbf{61} (2002), no.~2, 123--137.

\bibitem{Roytenberg:1998ku}
Dmitry Roytenberg and Alan Weinstein, \emph{{Courant algebroids and strongly
  homotopy Lie algebras}}, Letters in Mathematical Physics \textbf{46} (1998),
  no.~1, 81--93.

\bibitem{LetToWein}
Pavol {\v S}evera, \emph{{Letters to A. Weinstein}}, Available at
  \url{http://sophia.dtp.fmph.uniba.sk/~severa/letters/}.

\bibitem{Severa:2005vla}
\bysame, \emph{{Some title containing the words ``homotopy'' and
  ``symplectic'', e.g. this one}}, Travaux math\'ematiques. Fasc. XVI, Univ.
  Luxemb., Luxembourg, 2005, pp.~121--137.

\bibitem{Severa:2007vu}
\bysame, \emph{{$L_\infty$-algebras as first approximations}}, 26th Workshop on
  Geometrical Methods in Physics, Amer Inst of Physics, Melville, NY, 2007,
  pp.~199--204.

\bibitem{Severa:2001}
Pavol {\v S}evera and Alan Weinstein, \emph{{Poisson geometry with a 3-form
  background}}, Progress of Theoretical Physics. Supplement (2001), no.~144,
  145--154.

\bibitem{Sheng:2011vz}
Yunhe Sheng and Chenchang Zhu, \emph{{Higher Extensions of Lie Algebroids and
  Application to Courant Algebroids}},  (2011), Available at
  \url{http://arxiv.org/pdf/1103.5920}.

\bibitem{Sheng:2009un}
\bysame, \emph{{Semidirect products of representations up to homotopy}},
  Pacific Journal of Mathematics \textbf{249} (2011), no.~1, 211--236.

\bibitem{Xu95}
Ping Xu, \emph{{On Poisson groupoids}}, International Journal of Mathematics
  \textbf{6} (1995), no.~1, 101--124.

\bibitem{zhu2009kan}
Chenchang Zhu, \emph{{Kan replacement of simplicial manifolds}}, Letters in
  Mathematical Physics \textbf{90} (2009), no.~1, 383--405.

\end{thebibliography}
\bibliographystyle{amsplain-url}
\end{document}